 \documentclass[final,leqno,onefignum,onetabnum]{siamltex1213}
\usepackage{bm}

\usepackage{amsmath,amssymb,latexsym}

  \newcommand\bmu{\bm u}
    \newcommand\bmI{\bm I}
    \newcommand\bmt{\bm t}
      \newcommand\bmn{\bm n}
    \newcommand\bmg{ g}
  \newcommand\A{\bm A}
    \newcommand\D{\bm D}
      
   \newcommand\oscf{{\rm osc}(\bm f)}
    \newcommand\osckf{{\rm osc}_{K}(\bm f)}
       \newcommand\osct{{\rm osc}(\bm t)}
    \newcommand\oscet{{\rm osc}_{E}(\bm t)}
    \newcommand\bmv{\bm v}
      \newcommand\bmw{\bm w}
            \newcommand\bmf{\bm f}
      \newcommand\Ch{{\mathcal C}_h}
        \newcommand\Eh{{\mathcal E}_h}
            \newcommand\Sh{{\mathcal S}_h}
      
      \newcommand\V{{\bm H}_{D}^{1}(\Omega)}

      \newcommand\Vh{\bm V_{h}}
       \newcommand\Ph{P_{h}}
       \newcommand\intv{\tilde{\bmv}}

       \renewcommand\P{L^{2}(\Omega) }
      \newcommand\Fh{{\mathcal F}_h}
           \newcommand\Lh{{\mathcal L}_h}
           \newcommand\B{{\mathcal B}}
                \newcommand\F{{\mathcal F}}
      
      \renewcommand\div{{\rm div}\,}
      
      \newtheorem{remark}[theorem]{Remark}
      
\title{ON THE ERROR ANALYSIS OF STABILIZED FINITE ELEMENT METHODS FOR THE STOKES PROBLEM}

\author{Rolf Stenberg\thanks{Department of Mathematics and Systems Analysis,
Aalto University, P.O. Box 11100,  \hfill \break00076 Aalto, Finland
e-mail: 
(\email{rolf.stenberg@aalto.fi}). }  
 \and {Juha Videman}\thanks{CAMGSD/Departamento de Matem\'atica, Instituto Superior T\'ecnico, Universidade   de Lisboa, Av. Rovisco Pais 1, 1049-001 Lisboa, Portugal (\email{videman@math.ist.utl.pt}).}  
}

\begin{document}

\maketitle
\slugger{mms}{xxxx}{xx}{x}{x--x}

\begin{abstract} For a family of stabilized mixed finite element methods for the Stokes equations a complete a priori and a posteriori error analysis is given.\end{abstract}



\pagestyle{myheadings}
\thispagestyle{plain}

\section{Introduction} 
 Stabilization of mixed finite element methods for saddle point problems  \cite{HFB,Hughes-Franca,FH88,Douglas-Wang} is by now a well-established technique to design stable methods with finite element spaces which do not have to satisfy the so-called Babu$\check{\rm s}$ka--Brezzi condition. 
 The idea is to add a properly weighted residual of the momentum balance equation to the variational bilinear form.
   This resembles the least squares method and hence the formulation has often been called  the Galerkin least squares method.  This is, however, a somewhat misleading name since the formulation does not lead to a minimization problem.
 
 In the paper by Franca and Stenberg \cite{SIAM} a unified stability and error analysis for this class of methods was given. The error estimates were obtained under the assumption that the solution is regular enough, and so far a general analysis has been missing. 
 
 The purpose of this paper is to address this question. We will show that using 
 a technique proposed recently by Gudi 
 \cite{Gudi}  it is possible to derive quasi-optimal a priori estimates. This technique uses estimates  known from a posteriori error analysis.

 In addition to the a priori analysis, we discuss a posteriori estimates. Since the added stabilization term is exactly a weighted residual, the a posteriori analysis is very straightforward. Similar a posteriori estimates were given in [18]. Our analysis seems, however, more natural.

  The plan of the paper is as follows. In the next  sections we recall the continuous Stokes problem and its discretization by stabilized finite element methods. Then we present the new a priori analysis. We end by deriving the a posteriori error estimates. We will use well established  notation. In addition, we will use the shorthand  notation $A\lesssim B$   for: there exist a positive constant $C$, independent of the mesh parameter $h$, such that $A\leq CB$.

 \section{The Stokes problem}
 
We consider the Stokes equations for slow (or very viscous) steady fluid flow or equivalently, the equations of incompressible elasticity, which we normalize in such a way that $2\mu=1$, where $\mu$ is the  dynamic viscosity or first Lam\'e parameter, respectively. Let $\mathbf {div}$ denote the vector valued divergence applied to tensors and denote the symmetric velocity gradient/strain tensor by 
\begin{equation}
\D(\bmv)=\frac{1}{2}\big( \nabla \bmv +\nabla \bmv^{T}\big).
 \end{equation}
 Introducing the second order differential operator
\begin{equation}
\A\bmv =\mathbf {div}\D(\bmv), 
\end{equation}
the problem is: given $(\bmf, \bmt,g)$, find $(\bmu, p)$ such that
\begin{eqnarray} \label{momeq}
- \A \bmu+\nabla p&=& \bmf \quad \mbox{ in } \Omega,
\\
\div \bmu&=&\bmg \quad \mbox{ in } \Omega,
\\  \bmu&=&\bm0 \quad \mbox{ on }  \Gamma_{D},
\\ \big(\D(\bmu)-p\bmI\big)\bmn&=& \bmt\quad \mbox{ on }  \Gamma_{N} .
\end{eqnarray}
The domain $\Omega\subset I\!\!R^{d}, \ d=2,3, $ is assumed bounded and with a polygonal or polyhedral boundary. With the bilinear form
\begin{equation}\label{defB}
\B(\bmw,r;\bmv,q)=(\D (\bmw), \D( \bmv))-(\div \bmv, r)-(\div \bmw,q),
\end{equation}
and the linear form
\begin{equation}\label{defF}
\F( \bmv,q)=(\bmf,\bmv)+\langle \bmt,\bmv\rangle_{\Gamma_{N}}-(\bmg,q),
\end{equation}
we define the variational formulation.

\smallskip
\noindent 
\textsc{The continuous problem.} {\em 
 Find $(\bmu,p)\in \V\times \P$ such that }
 \begin{equation}\label{contprob}
\B(\bmu,p;\bmv,q)= \F(\bmv,q)\quad \forall(\bmv,q)\in \V\times \P.
 \end{equation}
Here,
$\V= {\bm H}^{1}(\Omega)\cap \{ \, \bmv \, \vert \, \bmv_{\vert \Gamma_{D}}=\bm 0\, \}.$
 
 The stability of this is a consequence of Korn's inequality
\begin{equation}\label{korn} C 
\Vert  \bm v\Vert_{1} \leq \Vert \D (\bm v)\Vert_{0}\leq \Vert  \bm v\Vert_{1}
\end{equation}
and the 
condition
\begin{equation}\label{clbb}
\sup_{\bmv \in \V}
 \frac{(\div \bmv, q)}{\Vert \bmv \Vert_{1}} \gtrsim   \Vert q\Vert_{0} \quad \forall q \in \P.
\end{equation}
Together they imply the stability:
 \begin{theorem} It holds
\begin{eqnarray}\label{contstab}
\sup_{(\bmv,q) \in \V\times \P}
\frac{\B(\bmw,r;\bmv,q)}{\Vert \bmv \Vert_{1}+\Vert q\Vert_{0}}&\gtrsim &(\Vert \bmw \Vert_{1}+\Vert r\Vert_{0})
\\ 
& &\forall( \bmw,r) \in \V\times \P. \nonumber
\end{eqnarray}
\end{theorem}
Classical mixed finite element methods are based on the variational formulation above posed in the finite element subspaces. By analogy with the continuous problem, the discrete spaces have to satisfy the  Babu$\check{\rm s}$ka--Brezzi condition, i.e. the discrete counterpart of \eqref{clbb}. The recent monograph \cite{bbf} contains the state of the art information on stable velocity--pressure pairs.

\section{Stabilized methods} We denote the piecewise polynomial finite element subspaces for the velocity and pressure  by $\Vh\subset \V$ 
and $\Ph\subset \P$, respectively. The underlying mesh is denoted by $\Ch$. As usual, the diameter of an element $K\in\Ch$, is denoted by $h_K$. Next, we define the bilinear and linear forms
 \begin{equation}\label{defsh}
\Sh(\bmw,r;\bmv,q)=\sum_{K\in \Ch}h_{K}^{2}(  - \A \bmw+\nabla r, - \A \bmv+\nabla q)_{K},\end{equation}
\begin{equation}\label{deflh}
\Lh(\bmv,q)= \sum_{K\in \Ch}h_{K}^{2}(  \bmf, - \A \bmv+\nabla q)_{K}.
\end{equation}
From the differential equation \eqref{momeq} it follows.
\begin{lemma} \label{s=l}If $\bmf \in \bm L_{2}(\Omega)$ it holds
\begin{equation}
\Sh(\bmu,p;\bmv,q)= \Lh(\bmv,q) \quad \forall (\bmv,q)\in \Vh\times \Ph.
\end{equation} 
\end{lemma}
\begin{proof} The differential equation \eqref{momeq} has to be interpreted in the sense of distributions. However, with the assumption $\bmf \in \bm L_{2}(\Omega)$ the sum $   - \A \bmv+\nabla q$ is in $  \bm L_{2}(\Omega)$ and hence both   
$
\Sh(\bmu,p;\bmv,q)$
 and  
$
 \Lh(\bmv,q)  $ are well defined and equal.
\end{proof}

Next, we define the forms
\begin{equation}\label{defbh}
\B_{h}(\bmw,r;\bmv,q)= \B(\bmw,r;\bmv,q)-\alpha \Sh(\bmw,r;\bmv,q)
\end{equation}
and
\begin{equation}\label{deffh}
\Fh(\bmv,q)=\F(\bmv,q)-\alpha\Lh(\bmv,q),
\end{equation}
where $\alpha$ is a positive constant less than   the constant $C_{I}$ in the following inverse inequality, which is valid in piecewise polynomial spaces with shape regular elements:
 \begin{equation}\label{inveq}
C_{I}\sum_{K\in\Ch}h_{K}^{2}\Vert \A\bmv\Vert_{0,K}^{2}\leq \Vert \D(\bmv)\Vert_{0}^{2}.
\end{equation}
The stabilized formulation is then the following.

 \textsc{The Finite element method. }{\em
 Find $(\bmu_h,p_h)\in \Vh\times \Ph$ such that }
 \begin{equation}\label{discprob}
\B_{h}(\bmu_{h},p_{h};\bmv,q)=\Fh(\bmv,q)\quad \forall(\bmv,q)\in \Vh\times \Ph.
 \end{equation}
 
 The consistency follows from Lemma \ref{s=l}. 
 
 \begin{theorem} Suppose that  $\bmf \in \bm L_{2}(\Omega)$.
 Then finite element method is consistent, in the sense that the exact solution $(\bmu,p)\in \V\times \P$  to \eqref{contprob} satisfies the discrete variational form
  \begin{equation}\label{consistency}
\B_{h}(\bmu,p;\bmv,q)=\Fh(\bmv,q)\quad \forall(\bmv,q)\in \Vh\times \Ph.
 \end{equation}
 \end{theorem}
 Next, we outline the main steps for analyzing   the stability of the formulation. For $(\bmv,q)\in \Vh\times \Ph$, the inverse inequality \eqref{inveq} and the assumption $0<\alpha < C_{I} $ give
\begin{eqnarray}\label{discstab1}\nonumber
\B_{h}(\bmw,r;\bmw,-r)&=&\Vert \D(\bmw)\Vert_{0}^{2}- \alpha
\sum_{K\in \Ch}h_{K}^{2} \Vert  \A \bmw\Vert^{2}_{0,K}
+\alpha
\sum_{K\in \Ch}h_{K}^{2} \Vert \nabla r\Vert^{2}_{0,K}
\\&\geq&
 \big(1-\frac{\alpha }{ C_{I}}\big)\Vert \D(\bmw)\Vert_{0}^{2} 
+\alpha
\sum_{K\in \Ch}h_{K}^{2} \Vert \nabla r\Vert^{2}_{0,K}
\\&\gtrsim&
 \big( \Vert \D(\bmw)\Vert_{0}^{2} 
+ 
\sum_{K\in \Ch}h_{K}^{2} \Vert \nabla r\Vert^{2}_{0,K}\big).\nonumber
\end{eqnarray}
As for the continuous problem the  stability for the velocity follows from  Korn's inequality \eqref{korn},
 whereas the stability of the pressure is  in the mesh dependent semi-norm \begin{equation}
\Big( \sum_{K\in \Ch}h_{K}^{2} \Vert \nabla r\Vert^{2}_{0,K}\Big)^{1/2}
 \end{equation}
 as a consequence of the added stabilization term. This gives stability for all pressures except the piecewise constants. In the case of continuous pressure approximations, all modes except the globally constant function 
  are stabilized. For discontinuous pressures the stable subspace is that of pressures orthogonal to the space of piecewise constants denoted by 
\begin{equation}
\Ph^{0}=\{ \ q\in \P \, \vert \ q\vert _{K}\in P_{0}(K) \ \forall K\in  \Ch\, \}.
\end{equation}The stabilization term has no influence on $P_{h}^{0}$. Hence the stability has to be based on the original bilinear form $\B$, i.e. we have to assume that the following discrete stability inequality is valid:
  \begin{equation}\label{conststab}
\sup_{\bmv \in \Vh}
\frac{(\div \bmv, q)}{\Vert \bmv \Vert_{1}}\gtrsim \Vert q\Vert_{0} \quad \forall q \in \Ph^{0}.
\end{equation}
The final stability estimate with the $L_{2}$-norm for the pressure is then proved using \eqref{discstab1} and 
\eqref{conststab} and a ``trick'', first introduced by Pitk\"aranta \cite{JPtrick}, and later applied for the Stokes problem by 
Verf\"urth \cite{RVtrick}. Our stability theorem is formulated as follows.

 \begin{theorem} Suppose that one of the following conditions is valid:
\begin{eqnarray}\nonumber
{\rm (i)}\quad && \Ph\subset {\mathcal C}^{0}(\Omega), 
\\
{\rm (ii)} \quad &&\mbox{the stability inequality \eqref {conststab} is valid.}
\nonumber
\end{eqnarray}
For 
$0< \alpha < C_{I}$  it then holds\begin{equation}\label{discstab2}
\sup_{(\bmv,q) \in \Vh\times \Ph}
\frac{\B_{h}(\bmw,r;\bmv,q)}{\Vert \bmv \Vert_{1}+\Vert q\Vert_{0}}\gtrsim  (\Vert \bmw \Vert_{1}+\Vert r\Vert_{0})\quad \forall( \bmw,r) \in \Vh\times \Ph.
\end{equation}
\end{theorem}
We emphasize the generality of the formulation. For continuous pressures all elements, triangles, quadrilaterals, tetrahedrons, prisms, hexahedrons and pyramids can be used, and also mixing them is allowed  provided the mesh is conforming. For discontinuous elements the only condition is that the stability estimate \eqref {conststab} is valid. In two dimensions this is true if the local element are  $[P_{2}(K)]^{2}$ and $[Q_{2}(K)]^{2}$ for triangles and quadrilaterals, respectively. In three dimensions the choices $[P_{3}(K)]^{3}$ and $[Q_{2}(K)]^{3}$ are sufficient for tetrahedrons and hexahedrons.

The following error estimate presented in the papers \cite{HFB,Hughes-Franca,FH88,SIAM,Nic} is  a direct consequence of the stability and consistency
\begin{eqnarray*}
 \Vert \bmu -\bmu_{h}\Vert_{1}+\Vert p-p_{h} \Vert_{0}  &\lesssim & 
 \inf_{\bmv \in \Vh} \Big(\Vert \bmu -\bmv\Vert_{1}+\big(\sum_{K\in\Ch}h_{K}^{2}\vert \bmu-\bmv \vert_{2,K}^{2}\big)^{1/2}
 \Big)
 \\
 & & +\inf_{q\in \Ph}\Big( \Vert p-q \Vert_{0}+ \big(\sum_{K\in\Ch}h_{K}^{2}\vert p-q \vert_{1,K}^{2}\big)^{1/2}
 \Big).
 \nonumber 
\end{eqnarray*}
The drawback of this estimate is that it requires that $\bmu\in \bm H^{2}(\Omega)$ and $p\in H^{1}(\Omega)$.
For less regular solutions the convergence was left open in the papers cited above.
In the following  we will amend this situation by using arguments introduced by Gudi~\cite{Gudi}. 

\section{A refined a priori error analysis}
First, we recall results from a posteriori error analysis \cite{RV,Vbook}. For an edge or face $E$ in the mesh, we denote by $\omega_{E}$ the union of all elements in $\Ch$ having $E$ as an edge or  a face. We define $\osckf$  by
\begin{equation}
\osckf=h_{K}\Vert \bmf -\bmf_{h}\Vert_{0,K},
\end{equation}
where $\bmf_{h}\in \Vh$ is the $L_{2}$-projection of $\bmf$. Similarly, we define 
\begin{equation}
\oscet= h_{E}^{1/2}\Vert \bmt -\bmt_{h}\Vert_{0,E},
\end{equation}
with $\bmt_{h}\in \Vh\vert_{\Gamma_{N}}$ being the $L_{2}$-projection. The global oscillation terms are defined  through

 \begin{equation}
\oscf^{2}=\sum_{K\in \Ch}\osckf^{2} \quad \mbox { and } \quad 
\osct^{2}=\sum_{E\subset \Gamma_{N}}\oscet^{2}.
\end{equation}
For an edge or face $E=\partial K\cap \partial K'$ the jump in the normal traction is
\begin{equation}
[\! [ (\D(\bmv)-q\bmI\big)\bmn]\!]\vert _E=   (\D(\bmv)-q\bmI\big)\big\vert_{K} {\bmn_{K}} -(\D(\bmv)-q\bmI\big)\big\vert_{K'}\bmn_{K'} .
\end{equation}
The following lower bounds are proved in \cite{RV,Vbook}.

\begin{lemma}\label{lowbound}
For all $ (\bmv,q)\in \Vh\times \Ph$ it  holds:
\begin{equation}\label{lowb1}
h_{K}\Vert \A \bmv-\nabla q + \bmf \Vert_{0, K }\lesssim \Vert \D (\bmu -\bm v)\Vert_{0,K}+ \Vert  p -q\Vert_{0,K}+\osckf
\quad \forall K\in\Ch.
\end{equation}
For $E$ in the interior of $\Omega$
\begin{equation}
h_{E}^{1/2}  \big\Vert   [\! [ (\D(\bmv)-q\bmI\big)\bmn]\!] \big\Vert_{0,E} \lesssim 
\Vert \D (\bmu -\bm v)\Vert_{0,\omega_{E}}+ \Vert  p -q\Vert_{0,\omega_{E}}+\sum_{K\subset \omega_{E}}\osckf
\end{equation}
and for $E\subset \Gamma_{N}$
\begin{eqnarray}
h_{E}^{1/2}  \big\Vert   (\D(\bmv)-q\bmI\big)\bmn-\bmt \big\Vert_{0,E} \lesssim 
 & & 
\Vert \D (\bmu -\bm v)\Vert_{0,\omega_{E}}+ \Vert  p -q\Vert_{0,\omega_{E}}
\\ && +\sum_{K\subset \omega_{E}}\osckf  +
\oscet.
\end{eqnarray}
\end{lemma}
Now we state the new error estimate. Note that   $\oscf$ is a higher order term.
 
 \begin{theorem} It holds
 \begin{equation}
 \Vert \bmu -\bmu_{h}\Vert_{1}+\Vert p-p_{h} \Vert_{0} \lesssim  
 \inf_{\bmv \in \Vh} \Vert \bmu -\bmv\Vert_{1}+\inf_{q\in \Ph}\Vert p-q \Vert_{0}+\oscf  .  
 \end{equation}
  \end{theorem}
 \begin{proof} Let $ (\bmv,q)\in \Vh\times \Ph$ be arbitrary. By the stability  estimate \eqref{discstab2} there exists $ (\bmw,r)\in \Vh\times \Ph$
 with 
 \begin{equation}\label{norma}
 \Vert \bmw \Vert_{1}+\Vert r\Vert_{0}=1
 \end{equation}
 and 
\begin{equation}\label{4.11}
\Vert \bmu_{h} -\bmv \Vert_{1}+\Vert p_{h}-q\Vert_{0}\lesssim 
\B_{h}(\bmu_{h}-\bmv ,p_{h}-q;\bmw,r).
\end{equation}  
Using  \eqref{discprob}, \eqref{deffh}, \eqref{contprob} and \eqref{defbh} yield
\begin{eqnarray}\label{est1}\nonumber
  \B_{h}(\bmu_{h}-\bmv ,p_{h}-q;\bmw,r)&= &
\B_{h}(\bmu_{h} ,p_{h},\bmw,r)-\B_{h}(\bmv ,q;\bmw,r)
\\&= & \nonumber
\Fh(\bmw,r)-\B_{h}(\bmv ,q;\bmw,r)
\\&=&
\F(\bmw,r)-\alpha\Lh(\bmw,r)-\B_{h}(\bmv ,q;\bmw,r)
\\&=&\nonumber \B(\bmu  ,p ;\bmw,r)-\alpha \Lh(\bmw,r)-\B_{h}(\bmv ,q;\bmw,r)
\\&=& \nonumber\B(\bmu  ,p;\bmw,r)-\alpha \Lh(\bmw,r)-\B(\bmv ,q;\bmw,r)
+\alpha \Sh(\bmv,q;\bmw,r)
\\&=& \B(\bmu  -\bmv,p-q ;\bmw,r)+\alpha \big(\Sh(\bmv,q;\bmw,r)-\Lh(\bmw,r) \big).\nonumber
\end{eqnarray}
From the boundedness of the bilinear form $\B$ and the normalization \eqref{norma}, we have
\begin{equation}\label{est2}
\B(\bmu  -\bmv,p-q ;\bmw,r)\lesssim
\big(
 \Vert \bmu -\bmv\Vert_{1}+ \Vert p-q \Vert_{0}\big).\end{equation}
 From the definitions \eqref{defsh} and \eqref{deflh}  we have
 \begin{equation}
 \Sh(\bmv,q;\bmw,r)-\Lh(\bmw,r)=\sum_{K\in \Ch}h_{K}^{2}(  - \A \bmv+\nabla q-\bmf, - \A \bmw+\nabla r)_{K}.
 \end{equation}
Cauchy--Schwarz   inequality then yields
  \begin{eqnarray*}\nonumber
&&\hskip -10 pt\vert \Sh(\bmv,q;\bmw,r)-\Lh(\bmw,r)\vert
\\&&\leq \Big(\sum_{K\in \Ch}h_{K}^{2}\Vert  - \A \bmv+\nabla q-\bmf\Vert_{0,K}^{2}\Big)^{1/2}
\Big(\sum_{K\in \Ch}h_{K}^{2}\Vert - \A \bmw+\nabla r\Vert_{0,K}^{2}\Big)^{1/2}
.
 \end{eqnarray*}
 By local inverse inequalities we have
  \begin{eqnarray*}\nonumber
\Big(\sum_{K\in \Ch}h_{K}^{2}\Vert - \A \bmw+\nabla r\Vert_{0,K}^{2}\Big)^{1/2}
&\leq& \Big(2\sum_{K\in \Ch}h_{K}^{2}\big(\Vert   \A \bmw \Vert_{0,K}^{2}+\Vert\nabla r\Vert_{0,K}^{2} \big)\Big)^{1/2}
\\&
\lesssim &( \Vert \bmw \Vert_{1}+\Vert r\Vert_{0}).
 \end{eqnarray*}
 Hence, \eqref{lowb1} gives
 \begin{eqnarray*}\label{est3}
\vert \Sh(\bmv,q;\bmw,r)-\Lh(\bmw,r)\vert&\lesssim &\Big(\sum_{K\in \Ch}h_{K}^{2}\Vert  - \A \bmv+\nabla q-\bmf\Vert_{0,K}^{2}\Big)^{1/2}
\\
 &\lesssim&
 \big(
 \Vert \bmu -\bmv\Vert_{1}+ \Vert p-q \Vert_{0}+\oscf\big)
. \nonumber
 \end{eqnarray*}
 The assertion now follows from \eqref{4.11}, \eqref{est1}, \eqref{est2} and \eqref{est3}.
\end{proof}
\begin{remark} The above estimates are also valid for the Douglas--Wang formulation {\rm \cite{Douglas-Wang}}, provided that 
the stabilizing term of inter-element pressure jumps  is dropped.

\end{remark}
\section{A posteriori estimates} For the a posteriori estimates we define the local    estimators
\begin{equation}
\eta_{K}^{2}= h_{K}^{2}\Vert \A \bmu_{h}-\nabla p_{h} + \bmf \Vert_{0, K }^{2}+\Vert \div \bmu_{h}-\bmg\Vert _{0,K}^{2}
\end{equation}
and
\begin{equation}
\eta_{E}^{2}=  \left\{\begin{array}{ll}
h_{E}   \big\Vert   [\! [ (\D(\bmu _{h})-p_{h}\bmI\big)\bmn]\!] \big\Vert_{0,E} ^{2}, \  \mbox{ when } E\subset \Omega,
\\
h_{E}  \big\Vert     (\D(\bmu _{h})-p_{h}\bmI\big)\bmn-\bmt\big\Vert_{0,E} ^{2}, \  \mbox{ when } E\subset \Gamma_{N}.
\end{array} \right.
\end{equation}
By $\Eh$ we denote the collection of edges/faces in $\Omega$ and on $\Gamma_{N} $. The global error estimator is then defined as
\begin{equation}
\eta^2= \sum_{K\in \Ch} \eta_K^2 +\sum_{E\in \Eh} \eta_E^2.
\end{equation}
Taking  $(\bmv,q)=(\bmu_{h},p_{h})$ in Lemma \ref{lowbound}  yields a local lower bound for the error. Now we will prove  the following upper bound.
\begin{theorem} It holds
 \begin{equation}
\big( \Vert \bmu -\bmu_{h}\Vert_{1}+\Vert p-p_{h} \Vert_{0}\big) \lesssim \eta.
 \end{equation} 
 \end{theorem}
\begin{proof} By the stability of the continuous problem \eqref{contstab},
there exists $ (\bmv,q)\in \V\times \P$
 with 
 \begin{equation}\label{stabnorma}
 \Vert \bmv \Vert_{1}+\Vert q\Vert_{0}=1
 \end{equation}
 and 
\begin{equation}\label{x}
\Vert\bmu- \bmu_{h}   \Vert_{1}+\Vert p- p_{h}\Vert_{0}
\lesssim 
\B (\bmu-\bmu_{h}  ,p-p_{h};\bmv,q).
\end{equation}  
Let $\intv \in \Vh $ be the Cl\'ement interpolant \cite{cle} of $\bmv$ for which we have the estimate\begin{equation}\label{clement}
\big( \sum_{K\in \Ch}h_{K}^{-2}\Vert \bmv-\intv \Vert_{0,K}^{2}+\sum_{E\in \Eh}h_{E}^{-1}\Vert \bmv-\intv \Vert_{0,E}^{2}\big)^{1/2}\lesssim \Vert \bmv \Vert_{1}\lesssim 1. \end{equation}
Choosing the pair $(\bmv,q)=(\intv,0)$ in the finite element formulation \eqref{discprob} and the consistency equation 
\eqref{consistency}, we get
\begin{equation}
\B_{h}(\bmu-\bmu_{h},p-p_{h},\intv,0)=0.
\end{equation}
Subtracting this from the right hand side in \eqref{x}, and using the definition of $\B_{h}$, we obtain
\begin{eqnarray}
\lefteqn{\B (\bmu-\bmu_{h}  ,p-p_{h};\bmv,q)
\nonumber= \B (\bmu-\bmu_{h}  ,p-p_{h};\bmv,q)-\B_{h}(\bmu-\bmu_{h},p-p_{h},\intv,0)}
\nonumber\\&=& \B (\bmu-\bmu_{h}  ,p-p_{h};\bmv-\intv,q)-
\  \alpha   \Sh(\bmu-\bmu_{h},p-p_{h};\intv,0).
\label{upp-1}
\end{eqnarray}
The first term above is estimated exactly as in the analysis of the standard mixed method \cite{RV,Vbook}, using  element by element integration by parts and the interpolation estimate \eqref{clement}. This results in 
\begin{equation}\label{upp1}  
\B (\bmu-\bmu_{h}   ,p-p_{h};\bmv-\intv,q)\lesssim \eta.
\end{equation}
Recalling   definition \eqref{defsh},  equation \eqref{momeq}, and using an inverse inequality together with  estimate 
\eqref{clement}, we get 
\begin{eqnarray*}\label{upp2}
\nonumber 
\big \vert\Sh(\bmu-\bmu_{h},p-p_{h};\intv,0)\big \vert &=&\big\vert \sum_{K\in \Ch}h_{K}^{2}( \bmf  + \A \bmu_{h}-\nabla p_{h}, - \A \tilde \bmv)_{K}\big \vert 
\\ & \leq & 
\big( \sum_{K\in \Ch}h_{K}^{2}\Vert  \bmf  + \A \bmu_{h}-\nabla p_{h}\Vert_{0,K} \big)^{1/2} 
\big( \sum_{K\in \Ch}h_{K}^{2}\Vert \A \tilde \bmv \Vert_{0,K} \big)^{1/2} 
\\ &\lesssim &\eta \Vert \tilde \bmv \Vert_{1}\lesssim \eta.
\nonumber
\end{eqnarray*}
The assertion now follows by combining  the above estimates.
\end{proof}

\begin{remark}
Let us finally note  that previous works on the a posteriori estimates for stabilized methods have mostly been confined to low order methods or to methods with stabilizing pressure jump terms, cf. 
{\rm \cite{Voh,Kay, Lee-Kim,Song-Cai,Zheng,Wang-Wang}}. 
As mentioned in the introduction, our estimates are the same as in {\rm \cite{Wang-Wang}}, but our analysis is more straightforward.
\end{remark}


\begin{thebibliography}{10}

\bibitem{bbf}
{\sc Daniele Boffi, Franco Brezzi, and Michel Fortin}, {\em Mixed finite
  element methods and applications}, vol.~44 of Springer Series in
  Computational Mathematics, Springer, Heidelberg, 2013.

\bibitem{cle}
{\sc Ph. Cl{\'e}ment}, {\em Approximation by finite element functions using
  local regularization}, RAIRO Analyse Num\'erique, 9 (1975), pp.~77--84.

\bibitem{Douglas-Wang}
{\sc Jim Douglas, Jr. and Jun~Ping Wang}, {\em An absolutely stabilized finite
  element method for the {S}tokes problem}, Math. Comp., 52 (1989),
  pp.~495--508.

\bibitem{FH88}
{\sc Leopoldo~P. Franca and Thomas J.~R. Hughes}, {\em Two classes of mixed
  finite element methods}, Comput. Methods Appl. Mech. Engrg., 69 (1988),
  pp.~89--129.

\bibitem{Nic}
{\sc Leopoldo~P. Franca, Thomas J.~R. Hughes, and Rolf Stenberg}, {\em
  Stabilized finite element methods}, in Incompressible computational fluid
  dynamics: trends and advances, Cambridge Univ. Press, Cambridge, 2008,
  pp.~87--107.

\bibitem{SIAM}
{\sc Leopoldo~P. Franca and Rolf Stenberg}, {\em Error analysis of {G}alerkin
  least squares methods for the elasticity equations}, SIAM J. Numer. Anal., 28
  (1991), pp.~1680--1697.

\bibitem{Gudi}
{\sc Thirupathi Gudi}, {\em A new error analysis for discontinuous finite
  element methods for linear elliptic problems}, Math. Comp., 79 (2010),
  pp.~2169--2189.

\bibitem{Voh}
{\sc Antti Hannukainen, Rolf Stenberg, and Martin Vohral{\'{\i}}k}, {\em A
  unified framework for a posteriori error estimation for the {S}tokes
  problem}, Numer. Math., 122 (2012), pp.~725--769.

\bibitem{Hughes-Franca}
{\sc Thomas J.~R. Hughes and Leopoldo~P. Franca}, {\em A new finite element
  formulation for computational fluid dynamics. {VII}. {T}he {S}tokes problem
  with various well-posed boundary conditions: symmetric formulations that
  converge for all velocity/pressure spaces}, Comput. Methods Appl. Mech.
  Engrg., 65 (1987), pp.~85--96.

\bibitem{HFB}
{\sc Thomas J.~R. Hughes, Leopoldo~P. Franca, and Marc Balestra}, {\em A new
  finite element formulation for computational fluid dynamics. {V}.
  {C}ircumventing the {B}abu{\v{s}}ka-{B}rezzi condition: a stable
  {P}etrov-{G}alerkin formulation of the {S}tokes problem accommodating
  equal-order interpolations}, Comput. Methods Appl. Mech. Engrg., 59 (1986),
  pp.~85--99.

\bibitem{Kay}
{\sc David Kay and David Silvester}, {\em A posteriori error estimation for
  stabilized mixed approximations of the {S}tokes equations}, SIAM J. Sci.
  Comput., 21 (1999/00), pp.~1321--1336.

\bibitem{Lee-Kim}
{\sc Hyung-Chun Lee and Kwang-Yeon Kim}, {\em A posteriori error estimators for
  stabilized {$P1$} nonconforming approximation of the {S}tokes problem},
  Comput. Methods Appl. Mech. Engrg., 199 (2010), pp.~2903--2912.

\bibitem{JPtrick}
{\sc Juhani Pitk{\"a}ranta}, {\em Boundary subspaces for the finite element
  method with {L}agrange multipliers}, Numer. Math., 33 (1979), pp.~273--289.

\bibitem{Song-Cai}
{\sc Lina Song, Yanren Hou, and Zhiqiang Cai}, {\em Recovery-based error
  estimator for stabilized finite element methods for the {S}tokes equation},
  Comput. Methods Appl. Mech. Engrg., 272 (2014), pp.~1--16.

\bibitem{RVtrick}
{\sc R{\"u}diger Verf{\"u}rth}, {\em Error estimates for a mixed finite element
  approximation of the {S}tokes equations}, RAIRO Anal. Num\'er., 18 (1984),
  pp.~175--182.

\bibitem{RV}
\leavevmode\vrule height 2pt depth -1.6pt width 23pt, {\em A posteriori error
  estimators for the {S}tokes equations}, Numer. Math., 55 (1989),
  pp.~309--325.

\bibitem{Vbook}
\leavevmode\vrule height 2pt depth -1.6pt width 23pt, {\em A posteriori error
  estimation techniques for finite element methods}, Numerical Mathematics and
  Scientific Computation, Oxford University Press, Oxford, 2013.

\bibitem{Wang-Wang}
{\sc Junping Wang, Yanqiu Wang, and Xiu Ye}, {\em Unified a posteriori error
  estimator for finite element methods for the {S}tokes equations}, Int. J.
  Numer. Anal. Model., 10 (2013), pp.~551--570.

\bibitem{Zheng}
{\sc Haibiao Zheng, Yanren Hou, and Feng Shi}, {\em A posteriori error
  estimates of stabilization of low-order mixed finite elements for
  incompressible flow}, SIAM J. Sci. Comput., 32 (2010), pp.~1346--1360.

\end{thebibliography}
\end{document}